\documentclass{amsart}
\usepackage{amsmath}%
\usepackage{amsfonts}%
\usepackage{amssymb}%
\usepackage{graphicx}
\usepackage{fullpage}
\usepackage{xcolor}
\usepackage{geometry}
\usepackage[hidelinks]{hyperref}
\usepackage{bm}

\newcommand{\la}{\lambda}

\newcommand{\R}{\mathbb{R}}
\newcommand{\C}{\mathbb{C}}
\newcommand{\N}{\mathbb{N}}

\newcommand{\T}{\mathbb{T}}

\newcommand{\comment}[1]{}

\newtheorem{theorem}{Theorem}[section]
\newtheorem*{theorem*}{Theorem}
\newtheorem*{proposition*}{Proposition}

\newtheorem{corollary}[theorem]{Corollary}

\newtheorem{lemma}[theorem]{Lemma}
\newtheorem*{lemma*}{Lemma}

\newtheorem{proposition}[theorem]{Proposition}
\newtheorem{remark}[theorem]{Remark}

\newcommand{\e}{e}

\begin{document}
\author[P. Charron]{Philippe Charron}
\address[Philippe Charron]{Universit\'e de Gen\`eve, D\'epartement de Math\'ematiques, 1205, Geneva, Switzerland}
\email{philippe.charron@unige.ch}
\author[F. Pagano]{Fran\c cois Pagano}
\address[Fran\c cois Pagano]{Universit\'e de Gen\`eve, D\'epartement de Math\'ematiques, 1205, Geneva, Switzerland}
\email{francois.pagano@unige.ch}
\title{On the concentration of the Fourier coefficients for products of Laplace-Beltrami eigenfunctions on real-analytic manifolds}
\date{\today}

\begin{abstract}
    On a closed analytic manifold $(M,g)$, let $\phi_i$ be the eigenfunctions of $\Delta_g$ with eigenvalues $\la_i^2$ and let $f:=\prod \phi_{k_j}$ be a finite product of Laplace-Beltrami eigenfunctions. We show that $\left\langle f, \phi_i \right\rangle_{L^2(M)}$ decays exponentially as soon as $\la_i > C \sum \la_{k_j}$ for some constant $C$ depending only on $M$. Moreover, by using a lower bound on $\| f \|_{L^2(M)} $, we show that $99\%$  of the $L^2$-mass of $f$ can be recovered using only finitely many Fourier coefficients.
\end{abstract}

\maketitle

%

\section{Introduction and main results}

Consider the set $\{\phi_i\}_{i \geq 0}$ of $L^2$-normalized Laplace-Beltrami eigenfunctions\footnote{In this article, all the estimates involve the square root of the eigenvalues. Therefore, we define $\la_i^2$ as the eigenvalues rather than the more common $\la_i$ in order to ease the notation.} on a real-analytic closed manifold $(M, g)$: $$-\Delta_g \phi_i = \lambda_i^2 \phi_i.$$ 
For $f \in L^2(M)$, let us consider the Fourier decomposition $$f=\sum c_i(f) \phi_i, \quad c_i(f) := \int_M f \phi_i.$$

We are interested in the following question: if $f = \phi_j \phi_k$, how do the coefficients $c_i(f)$ decay? In other words, what is the error in the approximation of $\phi_j\phi_k$ by a finite number of eigenfunctions?

 In \cite{Seeley}, Seeley proved that for an arbitrary real-analytic function $f$, there exist $C(f)$ and $D(f)$ such that $|c_i(f)| \leq Ce^{-D\la_i}$. Since the Laplace-Beltrami eigenfunctions are analytic on a real-analytic manifold, one can apply this result to $f=\phi_j \phi_k$ to deduce the exponential decay of  $c_i(\phi_j \phi_k)$. However, it is unclear how the constants $C, D$ depend on $\phi_j$, $\phi_k$ and $M$.
 
 Motivated by questions from number theory, Sarnak \cite{Sarnak} proved the following result: if $M$ is a $d$-dimensional closed real-analytic manifold of constant sectional curvature -1 and $f = \phi_{k_1}\phi_{k_2}\cdots \phi_{k_n}$, then there exist $B(d,n)$ and $ C(f)$ such that $c_i(f) \leq C \lambda_i^B\e^{-\frac{\pi}{2}\la_i}$. In dimension two, Bernstein and Reznikov \cite{BernsteinReznikov} later replaced the power loss with a logarithmic loss (see also  \cite{KrotzStanton}).
In the case of general compact real-analytic manifolds, Zelditch \cite{Zelditch} proved that $c_i(\phi_j^2) \leq C \e^{-R\la_i}$ where $R$  is strictly less than the maximal analytic Grauert tube radius, and $C$ depends on $j$, $M$ and $R$. 

These results describe the behavior of $c_i(f)$ when $f$ is a product of eigenfunctions as $i$ goes to infinity, but they do not describe what happens for smaller values of $i$. For example, in the case of the flat torus $\T^d$, $c_i(\phi_j \phi_k) =0$ if ${\la_i}>{\la_j}+{\la_k}$. A natural question to ask is the following: for a general analytic manifold, what are the conditions on $i$, $j$, and $k$ such that $c_i(\phi_j \phi_k)$ decays exponentially? 

In the case of smooth manifolds, some recent results indicate when $c_i(f)$ starts to decay when $f$ is a product of eigenfunctions. Motivated by numerical applications for electronic structure computing \cite{LuYing}, Lu, Sogge, and Steinerberger \cite{LuSoggeSteinerberger, LuSteinerberger, Steinerberger}, showed that for any $\epsilon>0$, $c_i(\phi_j \phi_k)$ has polynomial decay at any order as soon as $\la_i\geq \max({\lambda_j}, {\lambda_k})^{1+\epsilon}$. This result was recently improved by Wyman \cite{Wyman} with the weaker condition $\la_i\geq (2+\epsilon)\max({\lambda_j}, {\lambda_k}).$  To the best of our knowledge, it is an open problem whether exponential decay holds for $c_i(\phi_j \phi_k)$ in the smooth setting.

\subsection{Main results}
 
 Our first main result shows quantitatively the exponential decay of $c_i\!\biggl(\prod\limits_{j \leq n} \phi_{k_j}\biggr)$ in the real-analytic setting:  

\begin{theorem}\label{maintheorem}
    Let $(M, g)$ be a real-analytic closed Riemannian manifold of dimension $d$. Let $\{\phi_k\}$ be an orthonormal basis of Laplace-Beltrami eigenfunctions on $M$ with eigenvalues $\la_k^2$ in increasing order.
    
    Let $n \geq 2$  and consider a finite set of positive indices $\{{k_1}, \dots, {k_n}\}$. There exist $c(M)$ and $C(M)$ such that the following holds: for any $i > 0$,
    \begin{align}
        \left|\int_{M} \phi_i\prod_{j\leq n}    \phi_{k_j}\right| \leq \e^{-c\la_i} \prod_{j \leq n}  \e^{C\la_{k_j}} \, .
    \end{align}

    \end{theorem}

\comment{
The constants depend on the following factors:
\begin{enumerate}
    \item $C$ depends on the injectivity radius of $M$, on the maximal radius of convergence of the Taylor expansion of the metric at any point $p \in M$, and on the absolute value of the sectional curvatures of $M$.
    \item $c$ depends on the smallest radius of holomorphic extension of the eigenfunctions at any point $p \in M$, on the largest absolute value of the coefficients of the holomorphic extension of the Laplacian at any $p \in M$, on the injectivity radius of $M$, and on the maximal radius of convergence of the Taylor expansion of the metric at any point $p \in M$. 
    
    Another way of seeing it, which will be clear from the proof below, is that it corresponds to half the distance by which we can harmonically extend the product of the holomorphic extension of the eigenfunctions $\{\phi_{k_j}\}$.
    \item $D_1$ depends on the volume of the manifold, on the first non-zero eigenvalue $\lambda_1>0$, and on $c$.
    \item $D_2$ depends on the smallest radius of holomorphic extension of the eigenfunctions at any point $p \in M$. It also depends on the injectivity radius of $M$, and on the maximal radius of convergence of the Taylor expansion of the metric at any point $p \in M$.
\end{enumerate}
}

If we only consider indices $i$ large enough we get the following straightforward Corollary:

 \begin{corollary}\label{corollary}
     Let $(M, g)$ be a real-analytic closed Riemannian manifold of dimension $d$. Let $\{\phi_k\}$ be an orthonormal basis of Laplace-Beltrami eigenfunctions on $M$ with eigenvalues $\la_k^2$ in increasing order. Let $n\geq 2$ and consider a finite set of positive indices $\{{k_1}, \dots, {k_n}\}$. There exists $C_1(M)$ and $C_2(M)$ such that if $\la_i \geq  C_1\sum\limits_{j\leq n}   \la_{k_j} $, then 
\begin{align}
        \left|\int_{M} \phi_i\prod_{j\leq n} \phi_{k_j}  \right| \leq e^{-C_2 \la_i} \, .
    \end{align}
     \end{corollary}

The second result that we prove is a Lemma well-known to specialists: the $L^2$ norm of a product of Laplace-Beltrami eigenfunctions cannot be too small.

\begin{lemma}\label{propositionlowerboundbegin}
      Let $(M, g)$ be a smooth closed Riemannian manifold of dimension $d$. Let $\{\phi_k\}$ be an orthonormal basis of Laplace-Beltrami eigenfunctions on $M$ with eigenvalues $\la_k^2$ in increasing order. Let $n\geq 2$ and consider a finite set of positive integers $\{{k_1}, \dots, {k_n}\}$. There exists $C_3(M)$ and $C_4(M,n)$ such that:
     \begin{equation}
         \left\|\prod\limits_{j \leq n}\phi_{k_j} \right\|_{L^2(M)} \geq C_3   \prod_{j \leq n} e^{ -C_4 \la_{k_j}} \, .
     \end{equation}
 \end{lemma}

\begin{remark}
    We note that on $\mathbb S^2$, the $L^2$ norm of the product of  $\Re (x+iy)^k$, $\Re (x+iz)^k$ and $\Re (y+iz)^k$ decays exponentially fast in $k$.
\end{remark}


In our last result, we show, by combining Corollary \ref{corollary} and Lemma \ref{propositionlowerboundbegin}, that one can recover most of the $L^2$ mass of a product of Laplace-Beltrami eigenfunctions using only finitely many terms in the Fourier expansion:

\begin{theorem}\label{approxbyfinitenumber}
 Let $(M, g)$ be a real-analytic closed Riemannian manifold of dimension $d$. Let $\{\phi_k\}$ be an orthonormal basis of Laplace-Beltrami eigenfunctions on $M$ with eigenvalues $\la_k^2$ in increasing order. Let $n\geq 2$, consider a finite set of positive integers $\{{k_1}, \dots, {k_n}\}$. Denote $f:=\prod_{j \leq n}\phi_{k_j}$ and denote by $c_i:=\left\langle f, \phi_i\right\rangle_{L^2(M)}$ its Fourier coefficients.  There exists $C_5(M,n)$ such that $99\%$ of the $L^2$ mass of $f$ is concentrated in the following way:
 \begin{align}\label{approxfinitenumber2}
    \left\|\sum_{i \in A_n} c_i \phi_{i}\right\|_{L^2(M)} \geq 0.99 \left\|f\right\|_{L^2(M)} \quad \quad \mbox{ where }  A_n:=\left\{i: \la_i \leq C_5 \sum\limits_{j \leq n} \la_{k_j} \right\}.  
    \end{align}
   
\end{theorem}

\begin{remark}
    To fix the ideas, we chose $0.99$ on the right-hand side of \eqref{approxfinitenumber2}, but any other constant can be chosen. Note that it will yield a different value for $C_5$.
\end{remark}

Theorem \ref{maintheorem} complements the results cited above, however, our method of proof is quite different: as opposed to estimates on automorphic forms in  \cite{Sarnak}, \cite{BernsteinReznikov}, \cite{KrotzStanton} and microlocal estimates on the complex half-wave propagator in \cite{Zelditch}, we use elliptic estimates on harmonic functions and a quantitative version of Cauchy-Kovalevskaya's theorem to prove Theorem \ref{maintheorem}. As such, our method does not require the use of microlocal analysis. 

Lemma \ref{propositionlowerboundbegin} improves on a classical result of Donnelly and Fefferman \cite[Proposition 4.1]{Donnelly1988} which states that the supremum of a single eigenfunction cannot be too small on a fixed set. The main step in the proof of Lemma \ref{propositionlowerboundbegin} involves finding a large set where every eigenfunction in the product is not too small. To do so, we use a recent quantitative version of Remez inequality \cite[Lemma 4.2]{LogRemez}. 


\section{Acknowledgements}
We would like to thank Sasha Logunov for suggesting this problem to us. P.C. was funded by the SNSF. F.P. was funded by NCCR SwissMAP and by the SNSF. 
\section{Definitions and notation}\label{defsection}

\subsection{The manifold:} We will denote by $(M,g)$ or sometimes only by $M$, a closed real-analytic Riemannian manifold of dimension $d$. 
We will denote the metric on $M \times \R$ by $g'$. For $p \in M$, $ O_p \subset \R^d$ will denote a normal chart of $M$ such that the image of $p$ is $0$ and $\psi_p : M \to O_p$ will be the map associated with $O_p$.

Let $R_{inj}$ be the injectivity radius of $M$. We note that $R_{inj}>0$ since $M$ is a compact Riemannian manifold without boundary. 

\subsection{Multi-indices}
Let $\alpha \in \N^k$, with $k >0$. We denote by $\alpha_j$ the $j$-th component of $\alpha$. 

\subsection{Partial derivatives}

Let $l >0$, $k >0$. 
We define $\partial_{x_k}^l := \frac{\partial^l}{\partial {x_k}^l}$ and  $\partial_{z_k}^l := \frac{\partial^l}{\partial {z_k}^l}$. If $l=1$ we write $\partial_{x_k}$ or $\partial_{z_k}$. We also define $\partial_x^\alpha:= \partial_{x_1}^{\alpha_1}\partial_{x_2}^{\alpha_2}\ldots \partial_{x_k}^{\alpha_k}$ and $\partial_z^\alpha:= \partial_{z_1}^{\alpha_1}\partial_{z_2}^{\alpha_2}\ldots \partial_{z_k}^{\alpha_k}$.

\subsection{Polydisks} We will work with real and complex polydisks. For $k = d$ or $d+1$, we denote
   \begin{align}\label{defpolyclassical}
   PD_{\R^k}(R):=\{y \in \R^k: |y_i |< R, \quad 1\leq i\leq k \}.
   \end{align}

We will also use polydisks stretched in the last variable: 
\begin{align}\label{defpolystretched}
   PD_{\R^k}(R, \delta):=\{y \in \R^k: |y_i|< R,\quad 1\leq i\leq k-1, \quad |y_k|\leq R\delta \}.
\end{align} 

 We define $PD_{\C^k}(R)$ and $PD_{\C^{k}}(R, \delta)$ in the same manner.

\subsection{Power series}\label{powerseries}
We will say that a function $f$ has a power series expansion at $0$ that converges on a polydisk $PD_{\R^d}( r_0)$, $r_0>0,$ if for all $x \in PD_{\R^d}( r_0),$ 
$$
f(x) = \sum_{\alpha} \frac{\partial^{\alpha}_x f(0)}{\alpha!} x^{\alpha}
$$
with absolute convergence. It implies that $f$ has a holomorphic extension $f^\C$ to the polydisk $PD_{\C^d}(r_0)$ given by
$
f^\C(z) := \sum_{\alpha} \frac{\partial^{\alpha}_x f(0)}{\alpha!} z^{\alpha}
$
for all $z \in PD_{\C^d}(r_0)$, with absolute convergence. Furthermore, for any $x \in PD_{\R^d}( r_0)$ and $\beta \in \N^d$, $\partial_x^\beta f (x) = \partial_z^\beta f^\C (x)$.

\subsection{Laplace-Beltrami operator}

On a chart $O_p\subset \R^d$ of $M$ around $p \in M$, we denote the positively-defined Laplace-Beltrami operator  by $\Delta_g := -\text{div}_g\cdot \nabla_g$. Throughout the paper, if the chart $O_p$ is fixed we will use the convention $$\Delta_g := \sum\limits_{|\alpha| \leq 2} a_\alpha(x)\partial_x^\alpha, \quad x \in O_p.$$ 
On a fixed chart $O_p\times \R$ on the product manifold $M\times \R$ with metric $g'$, the positively-defined Laplace operator will be denoted by $$    \Delta_{g'} := \sum\limits_{|\alpha|\leq 2} a_{\alpha}(x) \partial^{\alpha}_x - \partial^2_{x_{d+1}}, \quad (x, x_{d+1}) \in O_p\times \R. $$
We will also consider the sequence $\phi_j$ of $L^2(M)$ normalized eigenfunctions of $\Delta_g$
\begin{align*}
    \Delta_g \phi_j = \lambda_j^2 \phi_j, \quad \|\phi_j\|_{L^2(M)}=1
\end{align*}
with eigenvalues arranged in increasing order.

\section{Proof of Theorem \ref{maintheorem}}

\subsection{Outline of the proof}

The proof of Theorem \ref{maintheorem} naturally decomposes into two parts. The first part (sections \ref{sectionradiusholoextencoeff} to \ref{backrealsection}) consists of finding a harmonic extension of a product of eigenfunctions with quantitative estimates. 

We note that the existence and uniqueness of the harmonic extension are guaranteed by the theorem of Cauchy-Kovalevskaya. However,  the classical real-valued version of this theorem does not give any (quantitative) information on the extension nor on the size of the domain where this extension is defined. On the other hand, there exists a lesser-known complex version of Cauchy-Kovalevskaya which includes quantitative estimates on both the supremum of this extension as well as on the size of the domain on which this extension is defined.

In Section \ref{sectionradiusholoextencoeff}, we show that for each $p \in M$, there is a normal chart at $p$ which includes a polydisk in which the coefficients of $\Delta_g$ have a bounded holomorphic extension. 

In Section \ref{sectionholoextenlaplaceefct}, we show that on each chart, there exists a possibly smaller polydisk in which the Laplace-Beltrami eigenfunctions have a bounded holomorphic extension.  We also estimate the supremum of the holomorphic extension of an eigenfunction in a fixed complex polydisk in terms of its eigenvalue. 

The results from these two previous sections now allow us to consider a complex Cauchy problem. In Section \ref{harmext}, we show, using a complex version of Cauchy-Kovalevskaya's theorem, the existence of a unique holomorphic harmonic extension of the product of eigenfunctions on a complex polydisk. We also find a quantitative estimate on the supremum of the extension and on the size of the polydisk in which this extension is well-defined.

In Section \ref{backrealsection}, we go back to the real setting and finish the construction. 

 In the second part (Section \ref{mainstepsection}), we use this harmonic extension together with Green's formula to estimate the Fourier coefficients.

\subsection{Radius of holomorphic extension of the coefficients of the Laplacian}\label{sectionradiusholoextencoeff}

    In this section, we show that for each $p \in M$, there is a chart $O_p$ containing a real polydisk of fixed radius (i.e., independent of $p$) and a complex extension of $\Delta_g$ on the complexification of this polydisk. More precisely, we have the following Proposition:

\begin{proposition}\label{propholomorphicextension}
        Let $(M,g)$ be a closed real-analytic manifold of dimension $d$. 
        There exists $R_1(M)$ such that for any $p \in M$, there exists a chart $O_p$ with the following properties:
        \begin{itemize}
            \item The polydisk $PD_{\R^d}(R_1)$ is fully contained in $O_p$.
            \item Let $\Delta_g=\sum_{|\alpha|\leq 2} a_{\alpha}\partial^{\alpha}_x$ be the local expression for the Laplace-Beltrami operator. Then, for any multi-index $|\alpha| \leq 2$, the function $a_\alpha(x)$, initially defined on $PD_{\R^d}(R_1)$, admits a bounded holomorphic extension $a_\alpha(z)$ to the complex polydisk $PD_{\C^d}(R_1)$.
        \end{itemize}
As a consequence, we can define the operator $L_g := \sum_{|\alpha|\leq 2} a_\alpha(z) \partial_z^\alpha$ as a differential operator with bounded analytic coefficients on $PD_{\C^d}(R_1)$.
        
    \end{proposition}

\begin{proof}
 Let $O_p=B_{R_{inj}}$ a normal coordinate chart. 
We immediately see that $PD_{\R^d}\left(\frac{R_{inj}}{\sqrt{d}}\right) \subset B_{R_{inj}}$. Since the metric is analytic, for each $p \in M$ the coefficients of the metric (in normal coordinates around $p$) can be expressed as a Taylor series which converges absolutely on a polydisk of radius $R'(p)$, but does not converge absolutely on any larger polydisk. We set 
\begin{align*}
    R':=\inf_{p \in M} R'(p)>0.
\end{align*}

 We know that each coefficient $a_\alpha$ is determined only by the matrix $g$, its derivatives and its inverse. Therefore, the radius of convergence of the Taylor series of $a_\alpha$ is the same as the radius of convergence of the coefficients of $g$. This implies that every coefficient $a_\alpha$ has a holomorphic extension to $PD_{\C^d}(R')$.

Denote  $R'' := \min\left(\frac{R_{inj}}{\sqrt{d}}, R'\right)$. By construction, $PD_{\R^d}(R'') \subset O_p$ and each coefficient $a_\alpha(x)$ has a holomorphic extension $a_\alpha(z)$ in $PD_{\C^d}(R'')$. Setting $R_1 := \frac{R''}{2}$ ensures that each $a_\alpha$ is bounded and holomorphic in $PD_{\C^d}(R_1)$. This completes the proof of Proposition \ref{propholomorphicextension}.
\end{proof}

\subsection{Holomorphic extension of Laplace eigenfunction}\label{sectionholoextenlaplaceefct}

On a given chart $O_p$, we wish to holomorphically extend the Laplace-Beltrami eigenfunctions into complex polydisks.

  \begin{proposition}\label{lemmasupcomp}
     Let $R_1$ be given by Proposition \ref{propholomorphicextension}. There exists  $R_2(M)<R_1$ and $C_6(M)$ and $ C_7(M)$ such that for any $p \in M$, every eigenfunction $\phi_k$ of $\Delta_g$ in a chart $O_p$ has a holomorphic extension in local coordinates on $ PD_{\C^d}(R_2) $ with the following estimate:
\begin{equation}\label{poitwiseestimatenovenocegood}
        \sup_{PD_{\C^d}(R_2)} |\phi_k| \leq C_6 \e^{R_1\lambda_k} (C_7\la_k)^{\frac{d-1}{2}} \, .
    \end{equation}
  \end{proposition}

\begin{remark}\label{remarkrtwocomplex}
We emphasize that the coefficients $a_{\alpha}$ of the operator $\sum_{|\alpha|\leq 2} a_{\alpha}(x) \partial^{\alpha}_x$ are holomorphic and bounded in the polydisk $PD_{\C^d}(R_2)$ and every eigenfunction $\phi_k$ has a bounded holomorphic extension to the polydisk $PD_{\C^d}(R_2)$.    
\end{remark}

The proof of Proposition \ref{lemmasupcomp} immediately follows from the next two Lemmas: 

\begin{lemma}[\cite{Lin}, Remark (d)]\label{holoextensionhormanderanalyticity}
 Let $R_1$ be given by Proposition \ref{propholomorphicextension}. There exists  $R_2(M)<R_1$ and $C_6(M)$ such that for any $p \in M$, every eigenfunction $\phi_k$ of $\Delta_g$ with eigenvalue $\la_k ^2$ in a chart $O_p$ has a holomorphic extension in local coordinates on $ PD_{\C^d}(R_2) $ with the following estimate:
    \begin{align}\label{estimateefunctionextension1main}
        \sup_{PD_{\C^d}(R_2)} |\phi_{k}| \leq C_6 \e^{R_1\la_k} \sup_{PD_{\R^d}(R_1)} |\phi_{k}|.
    \end{align}
\end{lemma}
Lemma \ref{holoextensionhormanderanalyticity} shows the holomorphic extension of Laplace eigenfunctions on analytic manifolds with an estimate. It has been proved by Donnelly and Fefferman  \cite{Donnelly1988} and by Lin \cite{Lin}, who used it to estimate the size of nodal sets of eigenfunctions on analytic manifolds. Here we refer to the simpler proof given by Lin.

\begin{lemma}[\cite{Donnelly}, Theorem 1.6]\label{lemmaforlinfinityefct}
    Let $\phi_k$ be an eigenfunction of $\Delta_g$ with eigenvalue $\la_k^2$. There exists $C_7(M)$ such that $\sup\limits_M |\phi_{k}|\leq  (C_7\la_k)^{\frac{d-1}{2}}$.
\end{lemma}
Different authors have proved Lemma \ref{lemmaforlinfinityefct} (\cite{Hormanderlinfinitybounds, Soggelinfinity}) but we refer to the proof of Donnelly \cite{Donnelly} since it only uses elliptic estimates and avoids the use of microlocal analysis.

\subsection{Harmonic extension of the product of the complexified eigenfunctions}\label{harmext}

 For $p \in M$, we recall that in Proposition \ref{propholomorphicextension}, we defined a complex extension of the Laplace-Beltrami operator in a complex neighborhood $PD_{\C^{d}}(R_2)$ of $PD_{\R^{d}}(R_2) \subset O_p$ by
\begin{align}\label{complexoperator}
    L_g:= \sum_{|\alpha|\leq 2} a_{\alpha}(z)\partial^{\alpha}_z \, . 
\end{align}
We now define on $PD_{\C^{d}}(R_2) \times \C$ the operator $L_{g'} := L_g - \partial^2_{z_{d+1}}$. 
\begin{remark}\label{remarkchoiceofp}
Note that the operator $L_{g'}$ depends on the choice of chart, and hence on $p$. 
\end{remark}

We want to solve the following Cauchy problem for some $ R_3 \leq R_2$ and for some $\delta >0$ both depending only on $(M,g)$ :

\begin{align}\label{cproblemprooflemmasixmain1}
\left\{
\begin{array}{rll}
L_{g'} u_p& = 0  &\mbox{ on }  PD_{\C^{d+1}}( R_3,\delta), \\
u_p &= \prod\limits_{j\leq n} \phi_{k_j} &\mbox{ on } \{z_{d+1}=0\},\\
\partial_{z_{d+1}} u_p & = 0 & \mbox{ on } \{z_{d+1}=0\}.
\end{array}
\right.    
\end{align}  

We also want an estimate on $u_p$ and $\partial_{z_{d+1}} u_p$ depending only on $M$ and on the eigenvalues $\lambda_{k_j}$, $j \leq n$. The main result of this section is the following Proposition:

\begin{proposition}\label{propositionexistholharm}
  Let $R_2$ be given by Proposition \ref{lemmasupcomp} and let $R_3:=R_2/4$. 
  \comment{For $p \in M,$ define $\delta_{0, p}$ by
  $$
  \delta_{0, p}:=2(2^{d+1}e)^2 \sup_{z\in PD_{\C^d}(2R_3)} \left(\sum_{|\alpha|\leq 2} |a_{\alpha}(z)| (2R_3)^{2-|\alpha|}\right)
  $$
  where $PD_{\C^d}(2R_3)$ is the complexification of $PD_{\R^d}(2R_3) \subset O_p$. We further define $\delta>0$ by 
  $$
  \delta_0:=\sup_{p \in M} \delta_{0,p}, \quad
\delta:=\sqrt{\frac{1}{\delta_0}}.
  $$
  }
There exists $\delta(M, R_3)>0$ and there exists a unique holomorphic solution $u_p$ to the Cauchy problem \eqref{cproblemprooflemmasixmain1}. Furthermore,  the following bounds on $u_p$ hold:

\begin{align}\label{boundsonu}
    \sup_{PD_{\C^{d+1}}(R_3, \delta)} |u_p| \leq C_8  \prod_{j \leq n}C_6\e^{R_1\la_{k_j}} (C_7\la_{k_j})^{\frac{d-1}{2}}, \quad \sup_{PD_{\C^{d+1}}(R_3/2, \delta)} |\partial_{z_{d+1}} u_p| \leq \frac{2}{\delta R_3} \sup_{PD_{\C^{d+1}}(R_3, \delta)} |u_p|,
\end{align}
 with $C_6$ and $C_7$ as in Proposition \ref{lemmasupcomp}, $R_1$ as in Proposition \ref{propholomorphicextension} and  $C_8:=\left(\frac{1}{2^{2d+1}e^2}+1\right)$.
\end{proposition}
\comment{
\begin{remark}
    Since the coefficients $a_{\alpha}$ depend smoothly on $p \in M$, and since $M$ is compact, we see that $\delta>0$ and independent of $p$.
\end{remark}
}

To prove Proposition \ref{propositionexistholharm}, we will use the following complex version of Cauchy-Kovalevskaya's Theorem.

\begin{theorem}\cite[Theorem 7.4.5]{Hormanderold}\label{cauchykowthmainproof}
Consider the polydisk $PD_{\C^{d+1}}(R, \delta)$ where $R, \delta >0.$
    Assume that the coefficients $a_\alpha$ in the differential equation
    \begin{equation}\label{eqnckformainproof}
        \sum_{|\alpha|\leq m} a_\alpha \partial^{\alpha}_z u = f
    \end{equation}
are analytic in $PD_{\C^{d+1}}(R, \delta)$ and that  $a_{(0,\ldots, 0, m)} =1$. If 
\begin{equation}\label{conditionforck}
2(2^{d+1} e)^m \sum_{\alpha \neq (0,\ldots,0,m)}R^{m-|\alpha|} \delta^{m-\alpha_{d+1}} |a_\alpha(z)| \leq 1, \quad z \in PD_{\C^{d+1}}(R, \delta),
\end{equation}
   and if $f$ is bounded and analytic in $PD_{\C^{d+1}}(R, \delta)$, then there is a unique analytic solution $u$ to equation \eqref{eqnckformainproof} in $PD_{\C^{d+1}}(R/2,\delta)$ satisfying the boundary conditions
    \begin{equation}
    \partial^j_{z_{d+1}} u=0, \quad 0\leq j <m, \quad z_{d+1}=0.
    \end{equation}
    We also have the estimate
    \begin{equation}
        \sup_{PD_{\C^{d+1}}(R/2,\delta)} |u| \leq 2(R\delta)^m \sup_{PD_{\C^{d+1}}(R,\delta)}|f| \, .
    \end{equation}
\end{theorem}

\begin{proof}[Proof of Proposition \ref{propositionexistholharm}] The proof will be done in five steps:

\begin{itemize}
\item[Step 1:] Find $R$ such that the Cauchy problem \eqref{cproblemprooflemmasixmain1} is well-defined. This is done in section \ref{hormandersection1}.
    \item[Step 2:] Transform the Cauchy problem \eqref{cproblemprooflemmasixmain1} in the form $L_{g'}\widetilde{u_p}=f$ with the correct boundary conditions to apply Theorem \ref{cauchykowthmainproof}. This is done in section \ref{sectionbdrycond}.
    \item[Step 3:] Find $\delta >0$ which only depends on $M$ such that condition \eqref{conditionforck} is fulfilled. We then apply theorem \ref{cauchykowthmainproof} to solve the new Cauchy problem. This is done in section \ref{sectiondelta}. 
     \item[Step 4:] Go back to the original Cauchy problem \eqref{cproblemprooflemmasixmain1}. This is done in section \ref{sectionreal}.
    \item[Step 5:] Find upper bounds on  $|u_p|$ and $|\partial_{z_{d+1}} u_p|$. This is done in section \ref{boundsonf}.   
\end{itemize}

    \subsubsection{Step 1: Finding a suitable $ R$}\label{hormandersection1}
By the definition of $R_2$ (see Remark \ref{remarkrtwocomplex}),  the operator $L_{g'}$ has holomorphic and bounded coefficients on $PD_{\C^{d+1}}(R_2,\delta)$ for any $\delta \geq 0$. Moreover, by Proposition \ref{lemmasupcomp}, complexified eigenfunctions are holomorphic and bounded on $PD_{\C^{d}}(R_2)$. Trivially, they are also bounded on   $PD_{\C^{d+1}}(R_2, \delta)$ for any $\delta\geq 0$ 
 as well since they only depend on the first $d$ variables. Theorem \ref{cauchykowthmainproof} therefore suggests to choose $R_3 = R_2/2$ for the Cauchy problem \eqref{cproblemprooflemmasixmain1}. However, we will need to use Cauchy's estimates to bound derivatives of $u_p$ in Section \ref{boundsonf}. If we choose a radius $r$ for the Cauchy problem \eqref{cproblemprooflemmasixmain1}, then to use Cauchy's estimates, we will need that the eigenfunctions are holomorphic and bounded on $PD_{\C^d}(4r).$  Since we only know that the eigenfunctions are holomorphic and bounded on $PD_{\C^d}(R_2)$, we will instead choose $ R_3:= R_2/4$ for the Cauchy problem \eqref{cproblemprooflemmasixmain1}.

\subsubsection{Step 2: Changing boundary conditions}\label{sectionbdrycond}
We rewrite the Cauchy problem \eqref{cproblemprooflemmasixmain1} in a way that allows us to use Theorem \ref{cauchykowthmainproof}. To this end, we note that if $u_p$ solves the Cauchy problem \eqref{cproblemprooflemmasixmain1},  then $\widetilde{u_p} = u_p-\prod\limits_{ j \leq n} \phi_{k_j}$ solves the auxiliary Cauchy problem
\begin{align}\label{newcproblemprooflemmasixmain}
\left\{
\begin{array}{rll}
L_{g'} \widetilde{u_p}& = -L_g \left( \prod\limits_{j \leq n} \phi_{k_j}\right) & \mbox{ on }  PD_{\C^{d+1}}(R_3,\delta), \\
\widetilde{u_p} &= 0 &\mbox{ on } \{z_{d+1}=0\},\\
\partial_{z_{d+1}} \widetilde{u_p} & = 0 & \mbox{ on } \{z_{d+1}=0\}.
\end{array}
\right.    
\end{align}

\subsubsection{Step 3: Finding a suitable $\delta$}\label{sectiondelta}

To use Theorem \ref{cauchykowthmainproof}, we need condition \eqref{conditionforck} to be satisfied. Recall that  $L_{g'} := L_g  - \partial_{z_{d+1}}^2 $. Note that for any index $\alpha$ appearing in the definition of $L_g$, $\alpha_{d+1} = 0$ (i.e., the coefficients of $L_g$ only depend on the first $d$ variables). Therefore, condition \eqref{conditionforck} reads
\begin{equation*}
2(2^{d+1} e)^2 \sum_{|\alpha|\leq 2} (2R_3)^{2-|\alpha|} \delta^{2} |a_\alpha(z)| \leq 1, \quad z \in PD_{\C^{d+1}}(2R_3).
\end{equation*}

Recall that by the definition of $R_3$ (see Step 1 above), every coefficients $a_\alpha(z)$ is holomorphic and bounded on $PD_{\C^{d+1}}(2R_3, \delta)$ for any $\delta \geq 0$. Therefore, we can make the following choice for $\delta$: for $p \in M,$ define $\delta_{0, p}$ by
  \begin{align}\label{deltazeropdef}
  \delta_{0, p}:=2(2^{d+1}e)^2 \sup_{z\in PD_{\C^d}(2R_3)} \left(\sum_{|\alpha|\leq 2} |a_{\alpha}(z)| (2R_3)^{2-|\alpha|}\right) ,   
  \end{align}
  where $PD_{\C^d}(2R_3)$ is the complexification of $PD_{\R^d}(2R_3) \subset O_p$. We further define 
  \begin{align}\label{deltazerosanspdef}
  \delta_0:=\sup_{p \in M} \delta_{0,p} \quad \text{and} \quad
\delta:=\sqrt{\frac{1}{\delta_0}}.    
  \end{align}
  
 Since $M$ is compact, we see that $\delta>0$. With these choices of $R_3$ and $\delta$, Theorem \ref{cauchykowthmainproof} ensures that the Cauchy problem \eqref{newcproblemprooflemmasixmain} has a unique holomorphic solution $\widetilde{u_p}$ on  $PD_{\C^{d+1}}(R_3, \delta)$. Moreover, the solution $\widetilde{u_p}$ satisfies the estimate
\begin{align}\label{firststepmaincknice}
    \sup_{PD_{\C^{d+1}}(R_3,\delta)} |\widetilde{u_p}| \leq  2(2R_3 \delta)^2 \sup_{PD_{\C^{d}}(2R_3)}\left|L_g \left( \prod\limits_{j \leq n} \phi_{k_j}\right)\right|. 
\end{align}
 The supremum on the right-hand side is indeed on $PD_{\C^{d}}(2R_3)$ since the eigenfunctions $\phi_{k_j}$ only depend on the first $d$ variables. 

\subsubsection{Step 4: Back to the original problem}\label{sectionreal}

We will now go back to our initial Cauchy problem \eqref{cproblemprooflemmasixmain1}. To this end, we define $u_p :=\widetilde{u_p}+ \prod\limits_{j \leq n} \phi_{k_j}$, where $\widetilde{u_p}$ solves \eqref{newcproblemprooflemmasixmain}. It is clear that $u_p$ solves 
\begin{align}\label{newcproblemprooflemmasixmainbackatitbaby}
\left\{
\begin{array}{rll}
L_{g'}  u_p& = 0 & \mbox{ on }  PD_{\C^{d+1}}(R_3,\delta), \\
 u_p &= \prod\limits_{j \leq n} \phi_{k_j} &\mbox{ on } \{z_{d+1}=0\},\\
\partial_{z_{d+1}}  u_p & = 0 & \mbox{ on } \{z_{d+1}=0\}.
\end{array}
\right.    
\end{align} 

We note that $u_p$ is the unique holomorphic solution to \eqref{newcproblemprooflemmasixmainbackatitbaby} since $\widetilde{u_p}$ is the unique holomorphic solution to \eqref{newcproblemprooflemmasixmain}. Moreover, by \eqref{firststepmaincknice}, the following estimate holds for $u_p$:
\begin{align}\label{finalsteponemainckestimateonu}
    \sup_{PD_{\C^{d+1}}(R_3, \delta)} |u_p| \leq 2(2R_3 \delta)^2 \sup_{PD_{\C^{d}}(2R_3)}\left| L_g\left(\prod_{1\leq k \leq n}\phi_{j_k}\right)\right|  + \sup_{PD_{\C^{d}}(2R_3)}\left|\prod_{1\leq k \leq n}\phi_{j_k}\right|.
\end{align}

\subsubsection{Step 5: The bounds on $u_p$}\label{boundsonf}
It remains to estimate the right-hand side of \eqref{finalsteponemainckestimateonu}. By Proposition  \ref{lemmasupcomp}, 
\begin{align*}\label{pointwiseefctrecallnewgood}
\sup_{PD_{\C^d}(R_2)} |\phi_{k_j}| \leq C_6 \e^{R_1 \la_{k_j} }  (C_7\la_{k_j} )^{\frac{d-1}{2}}.    
\end{align*}

Since $R_3 = R_2/4$, we get the following bound:

\begin{equation}\label{boundsofproduct}
     \sup_{PD_{\C^{d}}(4R_3)}\left|\prod\limits_{j \leq n}\phi_{k_j}\right| \leq \prod_{j \leq n} C_6 \e^{R_1 \la_{k_j}} \left(C_7\la_{k_j} \right)^{\frac{d-1}{2}}.
\end{equation}

We will use Cauchy's estimate to bound the derivatives of a holomorphic function:
\begin{lemma}[Cauchy's estimates]\label{cauchyformulaholomoprhicextensionmain}
    Let $k \geq 1$, let $R, \delta>0$ and let $f$ be  holomorphic and bounded in $PD_{\C^{k+1}}(2R, \delta)$. Then, for $\alpha \in \N^k$
    \begin{align}\label{cauchyformulafirstderivative}
    \sup_{PD_{\C^{k+1}}(R, \delta)} |\partial_z^\alpha f| \leq \alpha! R^{-|\alpha|} \delta^{-\alpha_{k+1}} \sup_{PD_{\C^{k+1}}(2R, \delta)} |f| .
    \end{align}
\end{lemma}

Estimate \eqref{cauchyformulafirstderivative} directly implies the following :
\begin{equation}\label{boundsderivatives2}
 \sup_{PD_{\C^{d}}(2R_3)}\left|\partial^{\alpha}_z\left(\prod\limits_{j \leq n}\phi_{k_j}\right)\right| \leq \alpha! ( 2R_3)^{-|\alpha|} \sup_{PD_{\C^d}(4R_3)} \left|\prod\limits_{j \leq n}\phi_{k_j}\right|.
 \end{equation}

Using \eqref{boundsderivatives2}, we get the following inequalities:

\begin{align}\label{firsttoestimanewnew}
    2(2R_3 \delta)^2\sup_{PD_{\C^{d}}(2R_3)}\left|L_g\left(\prod_{j \leq n}\phi_{k_j}\right)\right| & \leq 2(2R_3 \delta)^2\sup_{PD_{\C^{d}}(2R_3)}\sum_{|\alpha|\leq 2} |a_{\alpha}| \left|\partial^{\alpha}_z\left(\prod\limits_{j \leq n}\phi_{k_j}\right)\right| \nonumber \\
    &\leq 4(2R_3 \delta)^2 \left(\sup_{PD_{\C^{d}}(2R_3)} \sum_{|\alpha|\leq 2} |a_{\alpha}|  (2R_3)^{-|\alpha|}\right) \sup_{PD_{\C^d}(4R_3)} \left|\prod\limits_{j \leq n}\phi_{k_j}\right|.
\end{align}
By the definition of $\delta$ in \eqref{deltazerosanspdef} and using \eqref{firsttoestimanewnew},
\comment{
 Now recall the definition of $\delta$ given in \eqref{choicefordelta}:
$$
\delta_0:=2(2^{d+1}e)^2\sup_{p \in M} \sup_{z\in PD_{\C^d}(2R_3)} \left(\sum |a_{\alpha}(z)| (2R_3)^{2-|\alpha|}\right), \quad
\delta:=\sqrt{\frac{1}{\delta_0}}.
$$

Inserting this in \eqref{firsttoestimanewnew},} we obtain the following:

\begin{align}\label{boundsonlgu}
     2(2R_3 \delta)^2\sup_{PD_{\C^{d}}(2R_3)}\left|L_g\left(\prod_{j \leq n}\phi_{k_j}\right)\right|  \leq  \frac{1}{2^{2d+1}e^2} \sup_{PD_{\C^d}(4R_3)} \left|\prod\limits_{j \leq n}\phi_{k_j}\right| \nonumber \\
     \leq   \frac{1}{2^{2d+1}e^2}  \prod_{j \leq n} C_6 \e^{R_1 \la_{k_j}} (C_7\la_{k_j} )^{\frac{d-1}{2}},
\end{align}
the last inequality following from \eqref{boundsofproduct}. Combining bounds  \eqref{boundsofproduct} and  \eqref{boundsonlgu} into \eqref{finalsteponemainckestimateonu}, we finally obtain the estimate

\begin{equation}\label{boundsonutotal}
      \sup_{PD_{\C^{d+1}}(R_3, \delta)} |u_p| \leq  \left(\frac{1}{2^{2d+1}e^2}+1\right) \prod_{j \leq n} C_6 \e^{R_1 \la_{k_j}} (C_7\la_{k_j} )^{\frac{d-1}{2}}.
\end{equation}

This completes the upper bound \eqref{boundsonu} on $u_p$.
Finally, the bound \eqref{boundsonu} on $\partial_{z_{d+1}}u_p$ follows directly from Lemma \ref{cauchyformulaholomoprhicextensionmain}. This finishes the proof of Proposition \ref{propositionexistholharm}.
\end{proof}

\subsection{Reduction to the real setting}\label{backrealsection}

So far, for any $p \in M$, we have constructed a holomorphic function $u_p$ solving $L_{g'}u_p=0$ on a complex neighborhood of  $PD_{\R^{d+1}}(R_3,\delta)$ in $O_p \times \R$. We will use these functions to construct a global real solution $H$ to $\Delta_{g'}H=0$ on $M \times (-T,T)$ for some $T>0$. We will denote by $(p, t)$ a point in $M \times \R$.

\begin{proposition}\label{exisitenceandunicity}
    With $R_3$ and $\delta$ defined as in Proposition \ref{propositionexistholharm}, let us define $T:=\frac{R_3\delta}{2}$. Then, there exists a real-analytic solution $H$ to the following problem:

\begin{align}\label{cproblemprooflemmasixmain}
\left\{
\begin{array}{rll}
\Delta_{g'} H& = 0  &\mbox{ on }  M \times (-T,T), \\
H &= \prod\limits_{j \leq n}\phi_{k_j} &\mbox{ on } \{t=0\},\\
\partial_{t} H & = 0 & \mbox{ on } \{t=0\}.
\end{array}
\right.    
\end{align}  
Furthermore, we have the following bounds for $H$: 

\begin{align}\label{boundsonH}
    \sup_{M \times (-T, T)} |H| \leq C_8   \prod\limits_{j \leq n}C_6\e^{R_1\la_{k_j} } (C_7\la_{k_j} )^{\frac{d-1}{2}}, \quad \sup_{M \times (-T, T)} |\partial_{t} H| \leq  \frac{1}{ T} C_8   \prod\limits_{j \leq n}C_6\e^{R_1\la_{k_j} } (C_7\la_{k_j} )^{\frac{d-1}{2}}, 
    \end{align}
where $C_6, C_7, C_8, R_1$ as in Proposition \ref{propositionexistholharm} .
\end{proposition}

We will first construct this function $H$ on charts by restricting the holomorphic functions $u_p$ constructed in Section \ref{harmext} to the real setting. To this end, we will use the following elementary property of holomorphic functions and (finite-order) differential operators with holomorphic coefficients: 
\\

\underline{\textbf{Fact:}} Let $A$ be a finite subset of $\N^k$, let $R>0$ and let $b_\alpha^\C: PD_{\C^k}(R) \to \C$ be bounded holomorphic functions for $\alpha \in A$. Let $u^\C : PD_{\C^k}(R) \to \C$ be bounded and holomorphic as well. Let $b_\alpha^\R:= \left. b_\alpha^\C \right|_{ PD_{\R^k}(R)} $ and $u^\R := \left. u^\C \right|_{ PD_{\R^k}(R)}$. Then,
\begin{align}\label{pderestrictionreals}
\sum\limits_{\alpha \in A} b_\alpha^\C(z) \partial_z^\alpha u^\C =0 , \hspace{0.2cm} \forall z \in PD_{\C^k}(R) \quad  \implies \quad\sum\limits_{\alpha \in A} b_\alpha^\R(x) \partial_x^\alpha u^\R =0 , \hspace{0.2cm} \forall x \in PD_{\R^k}(R)   .  
\end{align}
This fact can easily be verified by expanding both Taylor series and comparing the coefficients.
\\

Now, let $p \in M$, $PD_{\R^{d+1}}(R_3,\delta) \subset O_p \times \R$ and $u_p: PD_{\C^{d+1}}(R_3,\delta) \to \C$ be the solution to the Cauchy problem \eqref{cproblemprooflemmasixmain1}  constructed using Proposition \ref{propositionexistholharm}. Let $H_p$ be the restriction of $u_p$ to $PD_{\R^{d+1}}(R_3,\delta) \subset O_p\times \R$.

By \eqref{pderestrictionreals}, $H_p$ solves the following real-analytic Cauchy problem with real-valued boundary conditions:
\begin{align}\label{cauchyproblemrealchart}
\left\{
\begin{array}{rll}
\Delta_{g'} H_p& = 0  &\mbox{ on }  PD_{\R^{d+1}}(R_3,\delta), \\
H_p &= \prod\limits_{j \leq n}\phi_{k_j} &\mbox{ on } \{x_{d+1}=0\},\\
\partial_{x_{d+1}} H_p & = 0 & \mbox{ on }  \{x_{d+1}=0\}.
\end{array}
\right.    
\end{align}  

Note that $H_p$ is real-valued by the real version of Cauchy-Kovalevskaya's Theorem (since the imaginary part of $H_p$ solves the same Cauchy problem \eqref{cauchyproblemrealchart} but is identically zero at the boundary).

Using again the real version of Cauchy-Kovalevskaya's Theorem, we note that the restriction of $H_p$ to any polydisk $Q \subset PD_{\R^{d+1}}(R_3,\delta)$ that intersects $\{x_{d+1}=0\}$ is the only analytic solution to the following Cauchy problem:

\begin{align}\label{cauchyproblemrealchart2}
\left\{
\begin{array}{rll}
\Delta_{g'} f& = 0  &\mbox{ on }  Q, \\
f &= \prod\limits_{j \leq n}\phi_{k_j} &\mbox{ on } \{x_{d+1}=0\},\\
\partial_{x_{d+1}} f & = 0 & \mbox{ on }  \{x_{d+1}=0\}.
\end{array}
\right.    
\end{align}  

Now, recall that for any $p \in M$, we defined $\psi_p : M \to O_p$ as the normal map associated with $O_p$. Let us take two points $p,q \in M$ such that $\psi_p^{-1}\big(PD_{\R^{d}}(R_3)\big) \cap \psi_q^{-1}\big(PD_{\R^{d}}(R_3)\big) \neq \emptyset$. If we transplant the solution $H_p$ from this intersection to a subset of $O_q \times \R$ through the transition map it is a real-analytic solution to the Cauchy problem \eqref{cauchyproblemrealchart2} in some polydisk $Q \subset PD_{\R^{d+1}}(R_3,\delta) \subset O_q \times \R$. By the uniqueness property for solutions to the Cauchy problem \eqref{cauchyproblemrealchart2}, $H_q$ and $H_p$ have to coincide on $Q$. By analyticity, the two functions coincide on $\psi_p^{-1}\big(PD_{\R^{d}}(R_3)\big) \cap \psi_q^{-1}\big(PD_{\R^{d}}(R_3)\big). $

This allows us to patch all the solutions $H_p$ into one global solution $H$ on $M \times [-R_3 \delta, R_3 \delta]$. We now choose $T :=R_3 \delta /2$, and remark that $H$ solves the Cauchy problem \eqref{cproblemprooflemmasixmain}.
Finally, the bounds \eqref{boundsonH} on $H$ follow from the bounds \eqref{boundsonu} on $u_p$, as the bounds on $u_p$ did not depend on the choice of of $p$. This completes the proof of Proposition \ref{exisitenceandunicity}.

\comment{\begin{remark}
    By the real version of Cauchy-Kovalevskaya's Theorem, $H_p$ is real-valued (since the imaginary part of $H_p$ solves the same Cauchy problem \eqref{cauchyproblemrealchart} but is identically zero at the boundary).
\end{remark} }

\subsection{The final step}\label{mainstepsection}

  We are now ready to finish the proof of Theorem \ref{maintheorem}.

    \begin{proof}[Proof of Theorem \ref{maintheorem}]
    Let $T$ and $H$ be defined as in Proposition \ref{exisitenceandunicity}.  Let $M_T := M \times [0,T]$ and denote $f_i(x,t):=\phi_i(x) e^{-\la_it}$ for $(x,t) \in M_T$. 
    
    Since both $H$ and $f_i$ are harmonic on $M_T$, Green's formula yields

\begin{align}\label{heuristicgreenformula}
    0=\int_{M_T} \left(f_i\Delta_{g'} H  - H \Delta_{g'} f_i \right)= \int_{\partial M_T} \left(f_i\partial_n H  - H \partial_n f_i\right) \, ,
\end{align}
where $\partial_n$ denotes the outwards normal derivative on $\partial M_T$. Since $M$ is a manifold without boundary, $\partial M_T = M \times \{0 \}  \bigcup M \times \{ T\}$ and therefore
\begin{align}\label{alltogether}
    \int_{M \times \{0\}} \left(f_i\partial_t H  - H \partial_t f_i\right) =\int_{M \times \{T\}} \left(f_i\partial_t H  - H \partial_t f_i\right) \, .
\end{align}

Since $H$ solves the Cauchy problem \eqref{cproblemprooflemmasixmain}, \eqref{alltogether} becomes
\begin{align*}
   \int_{M \times \{0\}} \la_i \phi_i \prod\limits_{j \leq n}\phi_{k_j}   = \int_{M \times \{T\}} \left(\phi_i e^{-T\la_i}\partial_t H  + H \la_i\phi_i e^{-T\la_i}\right). 
\end{align*}
Since $\lambda_i>0$ for all $i>0$, we obtain
\begin{align*}
    \left|\int_{M} \phi_i \prod\limits_{j \leq n}\phi_{k_j}   \right| \leq e^{-T\la_i} \frac{{\text{Vol}(M)}^{1/2}}{\la_i} \left(\|\partial_t H\|_{L^\infty(M_T)} \|\phi_i\|_{L^2(M)}   + \la_i\|H\|_{L^\infty(M_T)} \|\phi_i\|_{L^2(M)}\right) .
\end{align*}
Since $\phi_i$ is normalized in $L^2(M)$ and applying the bounds on $H$ from Proposition \ref{exisitenceandunicity}, we obtain

\begin{equation}\label{boundfinalalmost}
    \left|\int_{M}  \phi_i \prod\limits_{j \leq n}\phi_{k_j}   \right| \leq e^{-T\la_i} \frac{{\text{Vol}(M)}^{1/2}}{\la_i}  C_8  \left(\frac{1}{T}+\la_i\right)\prod_{j\leq n } C_6\e^{R_1\la_{k_j}} \left(C_7\la_{k_j}\right)^{\frac{d-1}{2}}  \, .
\end{equation}

 By noticing that for $x>0,$ $x^{(d-1)/2} \leq d! e^{x}$, by denoting $\lambda_1$ the first non-zero eigenvalue of $\Delta_g$ on $M$, and by recalling that eigenvalues are arranged in increasing order,  we obtain
\begin{equation}\label{boundfinal}
    \left|\int_{M}  \phi_i \prod\limits_{j \leq n}\phi_{k_j}   \right| \leq  {\text{Vol}(M)}^{1/2} C_8   \left(\frac{1}{T\lambda_1}+1\right)  \e^{-T \la_i} \prod_{j \leq n} C_6 d! \,e^{(R_1+C_7)  \la_{k_j} }\, .
\end{equation}

Now, since $C_6, C_8, T$ depend only on $M$, there exists $C_9(M)>0$ such that $${\text{Vol}(M)}^{1/2} C_8\left(\frac{1}{T \la_1}+1\right) C_6 d! \leq e^{C_9 \la_1} .$$ Since eigenvalues are labelled in increasing order, by choosing  $c := T$ and $C := R_1 + C_7 + C_9$, we finally obtain
\begin{align*}
        \left|\int_{M} \phi_i \prod\limits_{j \leq n}\phi_{k_j}  \right| \leq \e^{-c\la_i} \prod_{j \leq n}  \e^{C\la_{k_j}}\, ,
    \end{align*}

which completes the proof of Theorem \ref{maintheorem}.
 
\end{proof}

 \section{Remez lower bound}\label{Remezsec}

In this Section, we prove Lemma \ref{propositionlowerboundbegin}, on a lower bound for the $L^2$ norm of a product of Laplace-Beltrami eigenfunctions. To do that, we will find a set of fixed measure where all eigenfunctions appearing in the product will be bounded from below. This relies on estimates for Laplace-Beltrami eigenfunctions from \cite{Donnelly1988} and for the size of sub-level sets of harmonic functions from \cite{LogRemez}:

 \begin{theorem}\cite[Proposition 4.1]{Donnelly1988}\label{donnellymin}
     Let $(M,g)$ be a smooth compact manifold without boundary. Let $\phi_{k}$ be a Laplace-Beltrami eigenfunction on $M$ with eigenvalue $\lambda_k^2$. Then, there exists $R_4, C_9, C_{10}$ depending only on $M$ such that the following holds: for any $p \in M$,
     \begin{align*}
         \sup_{B(p,R_4)}|\phi_{k}| \geq \e^{-C_9 \la_k-C_{10}} \|\phi_{k}\|_{L^{\infty}(M)}.
     \end{align*}
 \end{theorem}

 \begin{theorem}\cite[Theorem 4.2 (ii)]{Donnelly1988}\label{donnellydoubling}
     Let $(M,g)$ be a smooth compact manifold without boundary. Let $\phi_{k}$ be a Laplace-Beltrami eigenfunction on $M$ with eigenvalue $\lambda_k^2$. Then, there exists $R_5, C_{11}, C_{12}>0$ and depending only on $M$ such that the following holds:  for any $p \in M$ and for any $r < R_5,$
     \begin{align*}
         N\big(\phi_{k}, B(p,r)\big):=\ln\left( \frac{\sup_{B(p,2r)}|\phi_{k}|}{\sup_{B(p,r)} |\phi_{k}|}\right) \leq C_{11} \la_k + C_{12}.
     \end{align*}
     $N\big(\phi_{k}, B(p,r)\big)$ is called the doubling index of $\phi_{k}$ in the ball $B(p,r)$. 
 \end{theorem}
The two previous results were proved by Donnelly and Fefferman in their work on the size of the nodal sets of Laplace-Beltrami eigenfunctions on real-analytic manifolds. We will also need the following Remez inequality due to Logunov and Malinnikova:

 \begin{theorem}\cite[Lemma 4.2]{LogRemez}\label{Remezlog}
     Let $A$ be a uniformly elliptic matrix with Lipschitz coefficients. Suppose that $\text{div}(A \nabla u) =0$ in $(20d) Q$, where $Q \subset \R^d$ is a cube and $\sup\limits_Q |u| = 1$. Let $N(u,Q):=\ln\left( \frac{\sup_{2Q}|u|}{\sup_Q |u|}\right)$ be the doubling index of $u$ in $Q$. Let $$E_a:=\left\{x \in \frac{1}{2} Q : \, |u(x)| < e^{-a} \right\}.$$ We have the following bound:

\begin{equation}\label{Remezeq}
    \mu_d( E_a) \leq C_{R} e^{-\frac{\beta a}{N}} s(Q)^d
\end{equation}
     where $\mu_d$ denotes the $d$-dimensional Lebesgue measure, $s(Q)$ denotes the sidelength of $Q$ and $C_{R}, \beta$ depend on $A$ and $d$ only.
 \end{theorem}

We will start the proof of Lemma  \ref{propositionlowerboundbegin}. The first step is to find a large set where an eigenfunction is large.

\begin{proposition}\label{propremez}
    Let $(M,g)$ be a smooth closed manifold and $\phi_k$ be a Laplace-Beltrami eigenfunction on $M$ with eigenvalue $\la_k^2$. Then, there exists a chart $O_p$ and a cube $Q \subset O_p$ of sidelength depending only on $M$ with the following properties: for any $n$ fixed, there exists $C_{14}, C_{15}>0$  depending only on $M$, $Q$ and $n$ such that for any $k$,
\begin{equation}
   \mu_d \Big(\left\{|\phi_{\la_k}|<  e^{-C_{14} \la_k-C_{15}} \right\} \cap Q/2\Big)  \leq \frac{1}{2n}  {\mu_d(Q/2)}.
\end{equation}
    
\end{proposition}

\begin{proof}
Let $n$ be fixed. 
Let $p \in M$ be fixed and $O_p \subset \mathbb R^d$ be a normal chart. Let us fix a cube $Q$ of sidelength $r$ such that $20(d+1)Q \subset O_p$ and that both $r < R_4$ and $2r\sqrt{d} < R_5$, where $R_4$ comes from theorem \ref{donnellymin} and $R_5$ from theorem \ref{donnellydoubling}. We also define $Q' := Q \times [-r/2,r/2]$.
\\

From now on, every constant that will appear in the proof of Proposition \ref{propremez} may depend on $M$ and $Q$. To ease notation, since  $M$ and $Q$ are fixed, we will not explicitly state the dependence of each constant.
\\

For any $k \in \N$, we define $h_k(x,y) := \phi_{\la_k}(x)e^{\la_k y}$. Then, $h_k$ solves $\text{div}(A\nabla h_k)=0$ in $20(d+1)Q'$. For any $x \in Q/2$ and any $y \in [-r/4,r/4]$, $|h_k(x,y)| \leq e^{\frac{r}{4} \la_k} |\phi_{k}(x)|$. This implies that for any $L>0$, we have the following inequality :
\begin{equation}\label{eqcubes}
    \mu_d\Big( \left\{ |\phi_{k}| <  e^{-\frac{r}{4}{\la_k}}L \right\} \cap Q/2\Big) \leq \frac{2}{r} \mu_{d+1}\Big( \left\{ |h_k| <L \right\} \cap Q'/2\Big) .
\end{equation}

Note that the cube $Q$, which is centered at 0 and of sidelength $r$, is included within the ball $B\left(0, r\sqrt{d}\right)$. Hence
$$
N(\phi_{k}, Q) \leq \ln\left(\frac{\sup_{B(0, 2r\sqrt{d})}|\phi_{k}|}{\sup_{B(0,r)}|\phi_{k}|}\right).
$$
Therefore, by applying Theorem \ref{donnellydoubling} $\lfloor \ln_2\left(\sqrt{d}\right) \rfloor+2$ times, we have that 
$
N(\phi_{k}, Q) \leq  C_{11} {\la_k} + C_{12},
$
and therefore
\begin{align}\label{nhkq}
    N(h_k, Q) \leq C_{16} {\la_k} + C_{17}.
\end{align}

Similarly, since $B(0,r)\subset Q$, it comes by applying Theorem \ref{donnellydoubling} $\lfloor\ln_2\left(\frac{R_4}{r}\right)\rfloor+1$ times that
$$
\sup_{Q}|\phi_{k}|\geq \sup_{B(0,r)}|\phi_{k}| \geq \sup_{B(0,R_4)}|\phi_{k}| e^{-C_{18}{\lambda_k}-C_{19}}.
$$
 Hence, by Theorem \ref{donnellymin} and since $\|\phi_{k}\|_{L^{\infty}(M)}\geq \frac{1}{Vol(M)^{1/2}}$ (as $\|\phi_{k}\|_{L^2(M)}=1$), we get
$
\sup_{Q}|\phi_{k}| \geq e^{-C_{20}{\lambda_k}-C_{21}}
$
 which implies that 
\begin{align}\label{lowerboundhkonqprime}
     \sup\limits_{Q'} |h_{k}| \geq  e^{-C_{22} {\la_k}-C_{23}}.
 \end{align}
 
The cube $Q'$ and the function ${h_k}(\sup\limits_{Q'} |h_{k}|)^{-1}$ both satisfy the conditions of theorem \ref{propremez}. Therefore, by \eqref{nhkq}, \eqref{lowerboundhkonqprime} and by theorem \ref{propremez},  there exists $\beta$ and $C_R$, such that for any $a >0$, we have the following inequality:

\begin{equation}\label{mudplusonehkwitha}
    \mu_{d+1} \Big( \left\{ |h_k| < e^{-a} e^{-C_{22} {\la_k}-C_{23}} \right\} \cap Q'/2\Big) \leq  C_{R} e^{\frac{-\beta a}{C_{16} {\la_k} + C_{17}} }. 
\end{equation}

We now set $a := C_{24}(n) {\la_k}$ where $C_{24}$ is large enough (and independent of $\lambda_k$) such that $$C_{R} e^{-C_{24}\frac{\beta  {\la_1}}{C_{16} {\la_1} + C_{17}} }\leq \frac{1}{2n} \mu_{d+1}(Q'/2) .$$ 

With this choice of $a$, we obtain using \eqref{mudplusonehkwitha} that
\begin{equation}\label{ineqcubes2}
    \mu_{d+1}\Big(  \big\{ |h_k| <  e^{-(C_{24}+C_{22}) {\la_k}-C_{23}} \big\} \cap Q'/2\Big) \leq \frac{1}{2n} \mu_{d+1}(Q'/2).
\end{equation}

Combining equations \eqref{eqcubes} and \eqref{ineqcubes2}, and noticing that $\mu_{d+1}(Q'/2) = r/2 \mu_{d}(Q/2)$, we obtain the following inequality:

\begin{equation}\label{ineqcubes3}
    \mu_{d}\Big(  \big\{ |\phi_{k}| <  e^{-(C_{24}+C_{22}+\frac{r}{4}) {\la_k}-C_{23}} \big\} \cap Q/2\Big) \leq \frac{1}{2n} \mu_{d}(Q/2) \, .
\end{equation}

Choosing $C_{14} := C_{24}+C_{22}+\frac{r}{4}$ and $C_{15} := C_{23}$, and noticing that they both only depend on $M,Q$ and $n$, we finally obtain 
\begin{equation*}
    \mu_{d}\Big(  \big\{ |\phi_{k}| <  e^{-C_{14} {\la_k}-C_{15}} \big\} \cap Q/2\Big) \leq \frac{1}{2n} \mu_{d}(Q/2) \, .
\end{equation*}

which completes the proof of Proposition \ref{propremez}.
\end{proof}

Now that we have found a small set where each eigenfunction is small, we can find a large set where all the eigenfunctions in the product are bounded from below.

\begin{corollary}\label{corrolaryremez}
    Let $(M,g)$ be a smooth closed manifold. Let $n \geq 1$. For any $0 < k_1 <  \ldots < k_n$, there exists a set $E$ of measure at least $C_{25}$ where $C_{25}>0$ is independent of $n$ and of the choice of $k_j$ and there exists $C_{26}(M,n)>0$ with the following property: for any $1 \leq j \leq n$, we have

    \begin{equation}
        \inf\limits_E |\phi_{k_j}| > e^{-C_{26} \la_{k_j}}.
    \end{equation}
    
\end{corollary}

\begin{proof}
Choose $O_p$ and $Q$ as in Proposition \ref{propremez}. By Proposition \ref{propremez}, we know that , 
$$
\mu_d \left(\bigcup_{j \leq n}\left\{|\phi_{k_j}|<  e^{-C_{14} \la_{k_j} -C_{15}} \right\} \cap Q/2\right)  \leq \frac{1}{2} {\mu_d(Q/2)},
$$
and therefore,
$$
\mu_d \left(\bigcap_{j \leq n}\left\{|\phi_{k_j}|\geq   e^{-C_{14} \la_{k_j} -C_{15}} \right\} \cap Q/2\right)  \geq \frac{1}{2} {\mu_d(Q/2)} .
$$

Then, denote 
$$
E':= \bigcap_{j \leq n}\left\{|\phi_{k_j}|\geq   e^{-C_{14} \la_{k_j} -C_{15}} \right\} \cap Q/2\, .
$$
 Set $E:= \psi_p^{-1}(E')$ and observe that $e^{-C_{14} \la_{k_j} -C_{15}}\geq e^{-C_{26}\la_{k_j}}$ where $C_{26}(M,n):=C_{14}+\frac{C_{15}}{\lambda_1}>0$. This finishes the proof of Corollary \ref{corrolaryremez}.
\end{proof}

We can now complete the proof of Lemma \ref{propositionlowerboundbegin}:

\begin{proof}[Proof of Lemma \ref{propositionlowerboundbegin}]
For $n$ fixed, let the set $E$  be as in Corollary \ref{corrolaryremez}. Let $\mu_g$ be the Lebesgue measure on $M$ induced by $g$. We have the following inequalities:
\begin{align*}
     \left\|\prod_{j \leq n}\phi_{k_j} \right\|_{L^2(M)} \geq  \left\|\prod_{j \leq n}\phi_{k_j} \right\|_{L^2(E)}
     \geq \mu_g(E)^{1/2}\, \inf\limits_E \left| \prod_{j \leq n}\phi_{k_j} \right|
     \geq \mu_g(E)^{1/2}\, \prod_{ j \leq n} e^{-C_{26} {\lambda_{k_j}}}
\end{align*}
where the last inequality follows from Corollary \ref{corrolaryremez}. We now set $C_{3}(M) :=  \mu_g(E)^{1/2}$ and $C_{4}(M,n) := C_{26}$ to finish the proof of Lemma \ref{propositionlowerboundbegin} .
\end{proof}

\section{The finite approximation}\label{finitesec}

In this Section, we prove Theorem \ref{approxbyfinitenumber} about the finite approximation in $L^2$ for a product of Laplace-Beltrami eigenfunctions.
\begin{proof}[Proof of Theorem \ref{approxbyfinitenumber}]
     Let $S:=\sum\limits_{i \geq 1} e^{-C_2 \la_i}$, where $C_2$ comes from Corollary \ref{corollary}. Note that $S<\infty$ by Weyl's asymptotic formula (see for instance \cite{chavel1984eigenvalues}).
     
     Let $C_5(M,n)\geq C_1$ be a large constant that will be defined later.   Denote $$A_n:=\left\{i: \la_i \leq C_5 \sum\limits_{j \leq n} \la_{k_j} \right\}$$ and let $A_n^c = \N \backslash A_n$ . Recall that we defined $c_i:=\int\limits_M \phi_i\prod\limits_{j \leq n}\phi_{k_j} .$ Since $\{\phi_i\}$ is an orthonormal basis of $L^2(M)$,
    \begin{align*}
        \left\|\prod\limits_{j\leq n} \phi_{k_j}-\sum_{i \in A_n} c_{i} \phi_{i}\right\|^2_{L^2(M)} &=\sum_{i \in A_n^c} |c_{i}|^2 \, ,\\
        &\leq \sum_{i \in A_n^c}  e^{-2C_2 \la_i} \,
    \end{align*}
   where the second inequality follows from Corollary \ref{corollary} and the fact that $C_5 \geq C_1$. Hence, using Lemma \ref{propositionlowerboundbegin}, we get that 
 $$
 \left\|\prod\limits_{j\leq n} \phi_{k_j}-\sum_{i \in A_n} c_{i} \phi_{i}\right\|^2_{L^2(M)} \leq \left\|\prod\limits_{j\leq n} \phi_{k_j}\right\|^2_{L^2(M)} \, \frac{1}{C_3^2} \prod_{j\leq n} e^{2C_{4}\la_{k_j}} \sum_{i \in A_n^c}  e^{-2C_2 \la_i} .
 $$
 
 By the definition of $A_n$ and $S$,
 \begin{align}\label{almostlastineqapprox}
    \left\|\prod\limits_{j\leq n} \phi_{k_j}-\sum_{i \in A_n} c_{i} \phi_{i}\right\|^2_{L^2(M)} & \leq \left\|\prod\limits_{j\leq n} \phi_{k_j}\right\|^2_{L^2(M)} \, \frac{1}{C_3^2}\prod_{j\leq n}e^{(2C_{4}-C_2 C_5) \lambda_{k_j}}  \, \sum_{i \in A_n^c}  e^{-C_2 \la_i} \nonumber \\
    & \leq \left\|\prod\limits_{j\leq n} \phi_{k_j}\right\|^2_{L^2(M)} \, \frac{S}{C_3^2}\prod_{j\leq n}e^{(2C_{4}-C_2 C_5) \lambda_{k_j}}  .
 \end{align}
 If $2C_{4}-C_2C_5\leq 0$, \eqref{almostlastineqapprox} becomes
 \begin{align}\label{eq48}
     \left\|\prod\limits_{j\leq n} \phi_{k_j}-\sum_{i \in A_n} c_{i} \phi_{i}\right\|^2_{L^2(M)}  \leq \left\|\prod\limits_{j\leq n} \phi_{k_j}\right\|^2_{L^2(M)} \, \frac{S}{C_3^2}e^{n(2C_{4}-C_2 C_5) \lambda_{1}}  .
 \end{align}

 Now, choose $C_5$ large enough so that $2C_{4}-C_2C_5\leq 0$ and
 $
 \frac{S}{C_3^2}e^{n(2C_{4}-C_2 C_5) \lambda_{1}} \leq \frac{1}{100^2}.
 $
With this choice of $C_5$, inequality \eqref{eq48} becomes
$$
\left\|\prod\limits_{j\leq n} \phi_{k_j}-\sum_{i \in A_n} c_{i} \phi_{i}\right\|_{L^2(M)}  \leq  \frac{1}{100} \left\|\prod\limits_{j\leq n} \phi_{k_j}\right\|_{L^2(M)},
$$
 
 which completes the proof of Theorem \ref{approxbyfinitenumber}.
\end{proof}

\bibliographystyle{plain} 
\bibliography{references}

\end{document}